\documentclass{article}

\title{A Rokhlin Lemma for Noninvertible Totally-Ordered Measure-Preserving Dynamical Systems}
\author{Adam Erickson}
\date{}

\usepackage{amssymb,amsmath,amsthm,amsfonts}
\usepackage{tikz}
\usepackage[hidelinks]{hyperref}

 \newtheorem{theorem}{Theorem}
 
  \newtheorem{proposition}[theorem]{Proposition}

 \newtheoremstyle{mydefstyle}%
     {}%
     {}%
     {}%
     {}%
     {\bf}%
     {}%
     { }%
     {\thmname{#1}\thmnumber{ #2}%
       \thmnote{: #3\addcontentsline{toc}{subsection}{#1 #2: {\it#3}}}. }
\theoremstyle{mydefstyle}
\newtheorem{definitionx}[theorem]{Definition}
\newtheorem{examplex}[theorem]{Example}
\newenvironment{definition}
   {\pushQED{\qed}\definitionx}
   {\popQED\enddefinitionx}
 \newenvironment{example}
   {\pushQED{\qed}\examplex}
   {\popQED\endexamplex}

\newcommand{\N}{\mathbb{N}}
\newcommand{\Q}{\mathbb{Q}}
\newcommand{\Z}{\mathbb{Z}}

\newcommand{\p}[1]{\left(#1\right)}
\newcommand{\sq}[1]{\left[#1\right]}
\newcommand{\set}[1]{\left\{#1\right\}}
\newcommand{\genset}[1]{\left\langle#1\right\rangle}
\newcommand{\abs}[1]{\left|#1\right|}

\newcommand{\defeq}{\overset{\text{def}}{=}}

\usepackage{mathrsfs}

\begin{document}
\maketitle

\begin{abstract}
Let $(X,\mathcal{F},\mu,T)$ be a not necessarily invertible
non-atomic measure-preserving dynamical system where the
$\sigma$-algebra $\mathcal{F}$ is generated by the intervals according to some total order.
The main result is that the classical Rokhlin lemma may be adapted to such a situation assuming
a slight extension of aperiodicity. This result is compared to previous noninvertible versions of 
the Rokhlin lemma.
\end{abstract}

\section{Introduction}
Rokhlin's lemma traces its history to the work of Vladimir Abramovitch Rokhlin, who utilized it in
studies of generators in ergodic theory \cite{Rokhlin63}. In its original context, it is a statement that in a
measure-preserving dynamical system $(X,\mathcal{F},\mu,T)$ with $T$ surjective, aperiodic, and invertible,
for any $n \in \N$ and $\varepsilon>0$, it is possible to find a
measurable subset $E$ such that each of the sets $T^j(E)$ for $j=0,\cdots,n-1$ are disjoint and these together cover all
but a subset of measure at most $\varepsilon$ of the space $X$; we say that such an $E$ induces
a $(n,\varepsilon)$-Rokhlin tower. For nearly 40 years, the Rokhlin lemma took this form
in most treatments (see \cite[pp.\ 71-72]{Halmos56}, \cite[p.\ 158]{Brown76}) but it entered the mathematical folklore that
minor adjustments to the proof would permit it to be applied to a non-invertible context.
 
A published proof of a Rokhlin lemma for non-invertible endomorphisms
does not seem to have manifested until 2000, when 
Stefan-Maria Heinemann and Oliver Schmitt published a proof in the context when $X$ is a separable metric space
with $T$ measure preserving, aperiodic and surjective (but not necessarily invertible or even forward-measurable) that 
there exists a set $E$ inducing an $(n,\varepsilon)$-Rokhlin tower, in the sense that $\mu(E \cap T^{-j} E) = 0$ for all
$j=1,2,\cdots,n-1$ and the sets $T^{-j} E$ cover all but a measure $\varepsilon$ subset of $X$.

We wish to highlight the flexibility of the proof strategy utilized by Heinemann and Schmitt 
in \cite{Heinemann00} by adapting the Rokhlin lemma
to a similar but distinct setting, that of systems generated by the intervals according to some total ordering with
non-invertible dynamics, in the hopes that this will inspire further investigation into conditions and contexts where
other analogous versions of Rokhlin's lemma hold.

In section \ref{sec3} we compare the variants of Rokhlin's lemma produced by Heinemann and Schmitt and which we prove in this paper,
and also compare our result to a similar but distinct version of Rokhlin's lemma
obtained by Avila and Candela \cite{Avila16} which applies to $\N^d$-actions
on (nonatomic) standard probability spaces.

For those so inspired we 
also suggest Avila and Bochi \cite{Avila06} which contains a distinct approach from \cite{Heinemann00} in proving
an interesting variation of the Rokhlin lemma for compact manifolds with $C^1$-endomorphism dynamics against a normalized Lebesgue measure
(which is not necessarily preserved by the dynamics).

\section{Rokhlin lemma}

\begin{definition}
Suppose $X$ is totally ordered by $<$. We define an \textbf{interval} in $X$ to be any set $J \subseteq X$ such that if $a , c \in J$, and 
$a < b < c$ for some $b \in X$, then $b \in J$ also.
\end{definition}
We utilize standard interval notation; in particular, we use the symbols $-\infty$ and $\infty$ to denote the absence of a lower bound
or upper bound. We also introduce the notation $\genset{a,b}$ to denote the closed interval between $a$ and $b$ regardless of their
relative order; that is, $\genset{a,b} = [ \min(a,b) , \max(a,b) ]$. When $a=b$, we let it be a singleton set $\genset{a,a} = \set{a}$.

\begin{definition}
Let $\Sigma = (X,\mathcal{F}(<),\mu,T)$ be a probability 
measure-preserving dynamical system such that $\mathcal{F}(<)$ (or simply $\mathcal{F}$)
is the $\sigma$-algebra
generated by the intervals according to some total order $<$ on $X$, $T$ is a measurable endomorphism,
and $\mu$ is non-atomic. We call $\Sigma$ a 
\textbf{shuffleable system}.
\end{definition}
Shuffleable systems emanate from the fact that they are
the most general class of measure-preserving dynamical system in which questions of shuffling in the 
sense of Persi Diaconis's geometric model of the Gilbert-Shannon-Reeds model of the riffle shuffle (see \cite{Diaconis})
may be studied.
Shuffleable systems include classical examples such as the doubling map and the one-sided 
full shift $\Sigma_2^+$, but the 
class of shuffleable systems is general enough to also include interval maps with non-atomic preserved measures as
well as systems on $\mathcal{I} = [0,1] \setminus \Q$, as well as full-measure subsets of 
standard probability spaces, which is essential as we will see in Section 3.

We note as a warning that, unlike with Lebesgue-Stieljes measures on the real line and standard probability spaces,
shuffleable spaces do not necessarily play well with the topology of open intervals.
\begin{example}
The measure space $(X=[0,1] \times (0,1),\mathcal{F}(<_{\text{lex}}),m_1)$, where 
$x <_{\text{lex}} y$ if $x_1<y_1$ or if both $x_1=y_1$ and $x_2<y_2$, and where 
$m_1$ is the probability measure induced on $X$ by the premeasure on intervals
given by $m_1(x,y] = y_1-x_1$, is a shuffleable space. In particular, if $\pi : X \to [0,1]$ is
the projection $\pi(x) = x_1$, and $m$ the Lebesgue measure, then for any interval $J \subseteq [0,1]$, 
$m_1(\pi^{-1}(J)) = m(J)$. However, if $F \subseteq [0,1]$ is a non-measurable subset of $[0,1]$, then 
$\pi^{-1}(F)$ is a non-measurable subset of $X$. Also, every set $\{x_1\} \times (0,1) \subseteq X$
is open, and so $\pi^{-1}(F) = \bigcup_{x_1 \in F} \{x_1\} \times (0,1)$ is open in 
the $<_{\text{lex}}$-order topology, yet non-measurable in $X$.
\end{example}

For a shuffleable system, we define $F_\mu(x) := \mu(-\infty,x]$, the distribution function of $\mu$.

%
\begin{proposition}\label{prop00}
If $(X,\mathcal{F}(<),\mu)$ is a non-atomic 
probability space with $\mathcal{F}$ the $\sigma$-algebra generated by the intervals according to the total order $<$, then for
any open interval $J = (a,b) \subseteq [0,1]$, there exists some point $x_J \in X$ such that 
\[ F_\mu(x_J) = \mu(-\infty,x_J] \in J. \]
\end{proposition}
\begin{proof}
By way of contradiction, suppose no value in $(a,b) \subseteq [0,1]$ is the measure of any 
left-ray $(-\infty,x]$; that is, $F_\mu(X) \cap (a,b) = \emptyset$.
Then the sets 
\[ L = \bigcup \set{ (-\infty,y] : y \in Y, F_\mu(y) = \mu(-\infty,y] \le a} \]
and
\[ U = \bigcap \set{ (-\infty,y] : y \in Y, F_\mu(y) = \mu(-\infty,y] \ge b}\]
are monotone unions and intersections of intervals, and thus are intervals themselves,
and $\mu(L) \le a < b \le \mu(U)$, so that $I = U \setminus L$ is also an interval with positive measure 
$\mu(I) = \mu(U) - \mu(L) \ge b-a>0$.

We note, however, that $U$ contains at most one element $y'$ such that $\mu(-\infty,y'] \ge b$; if there were two, $y_1,y_2 \in U$, then
$y_1 <y_2$ (without loss of generality) and then $U \subseteq (-\infty,y_1]$ cannot contain $y_2$, a contradiction. Thus, for every $x \in 
U \setminus \set{y'}$, $F_\mu(x) = \mu(-\infty,x] < b$. However, by assumption, no element $x$ satisfies $a < F_\mu(x) < b$, so this implies
further that $F_\mu(x) \le a$, so that $x \in L$. Consequently, $I = U \setminus L$ is a set containing at most one element, yet possessing
positive measure. If $I = \set{y'}$ then $I$ is atomic; if $I = \emptyset$, then $\mu(\emptyset) > 0$ is absurd. We arrive at a contradiction; 
there must exist $x \in X$ such that $F_\mu(x) \in (a,b)$.
\end{proof}

Following the general construction of Stefan Heinemann and Oliver Schmitt \cite{Heinemann00}, we require some analogue to the metric structure utilized; 
we resolve this by utilizing a pseudometric,
the measure of the interval between two points. However, we must earn permission to utilize 
a pseudometric in place of a metric by strengthening the aperiodicity condition.
\begin{definition}
We say $T$ is \textbf{$n$-nigh aperiodic} if 
$\mu\set{ x : \mu\genset{x,T^n x} = 0 } = 0$, and that $T$ is \textbf{nigh aperiodic} if it is $n$-nigh aperiodic for every $n \ge 1$.
\end{definition}

We define nigh aperiodicity in this way to highlight intervals of the form $\genset{x,T^n x}$ which will be key to our proof of
Rokhlin's lemma, but the measure therein has a much more familiar form in terms of the distribution function $F_\mu$. In particular,
\begin{align*}
\mu\!\genset{a,b} 
&=
\begin{cases}
\mu[a,b] = F_\mu(b) - F_\mu(a) & \text{if } a<b \\
\mu[b,a] = F_\mu(a) - F_\mu(b) & \text{if } b<a \\
\mu\{a\} =\mu\{b\}= 0 &\text{if } a=b
\end{cases}\\
&= \abs{F_\mu(b) - F_\mu(a)}.
\end{align*}
In particular this shows that $\set{x : \mu\!\genset{x,T^n x} = 0}$ is a measurable set, since it is the preimage of $0$ for
the map $\abs{F_\mu(T^{n} x) - F_\mu(x)}$, and thus nigh aperiodicity is well-defined. This also permits us to view the nigh
aperiodicity condition as a relation between the distribution function $F_\mu$ and its orbit in $L^1(X)$
under the operator $U_T$ defined
by $U_T f (x) = f(T(x))$: $T$ is nigh aperiodic iff for every $n$,
\[
\mu\{ x \in X : F_\mu(x) = U_T^n F_\mu(x) \} = 0.
\]

\begin{proposition}\label{prop1}
Suppose $(X,\mathcal{F}(<),\mu,T)$ is a nigh aperiodic shuffleable system. For every $n$ there exists $E \in \mathcal{F}$ of 
positive measure such that $\mu(E \cap T^{-n} E) = 0$.

Further, for any set $U$ of positive $\mu$-measure such that $\mu(U \triangle T^{-1} U) = 0$, 
and $\nu = \mu|_U$, there exists a set $E$ of positive $\nu$-measure such that 
$\nu(E \cap T^{-n} E) = 0$, so that $E' = U \cap E$ is a positive $\mu$-measure set satisfying $\mu(E' \cap T^{-n} E') = 0$.
\end{proposition}
\begin{proof}
By Proposition~\ref{prop00}, we know that every open subset in $[0,1]$ contains some value $r$ which is the value
of $F_\mu(x_r) = \mu(-\infty,x_r]$ for some $x_r \in X$, so that we can find a countable dense subset $S$ of $[0,1]$ of values 
in the range of $F_\mu$.
For each $r \in S$, we select one such $x_r$, and collect them together in the countable set
$Q\subseteq X$.

We now define a collection of intervals with endpoints in $Q$ which have measure no greater than some fixed $\varepsilon>0$. 
Let us define
\[ 
\mathscr{R} = \set{(-\infty,b], [a,b], [a,\infty): a,b \in Q}
\]
and from this we refine to
\[
\mathscr{R}_\varepsilon = \set{I \in \mathscr{R} : \mu(I) < \varepsilon}.
\]
$\mathscr{R}$ is in bijection with the disjoint union $Q \sqcup Q^2 \sqcup Q$, and thus it is countable, and $\mathscr{R}_\varepsilon$ inherits
countability also.

For every $\varepsilon>0$, we know $\mathscr{R}_\varepsilon$ covers $X$; for any $x \in X$ with $0<F_\mu(x)< 1$, we know there are $a,b \in Q$
such that 
\[
F_\mu(x) - \frac{\varepsilon}{2}
<
F_\mu(a)
<
F_\mu(x)
< 
F_\mu(b)
<
F_\mu(x) + \frac{\varepsilon}{2} .
\]
Thus, in particular, 
\[
\mu[a,b] = F_\mu(b) - F_\mu(a) < \varepsilon .
\]
By monotonicity, since $F_\mu(a) < F_\mu(x) < F_\mu(b)$, $x \in [a,b]$.
On the other hand, if $F_\mu(x) = \mu(-\infty,x]=0$, then there exists $b$ such that $0 < F_\mu(b)= \mu(-\infty,b] < \varepsilon$, 
and so again $x \in (-\infty,b]$, and if $F_\mu(x) = \mu(-\infty,x]=1$, then there exists $a$ such that
$1 > F_\mu(a) = \mu(-\infty,a] > 1-\varepsilon$, so that
$x \in [a,\infty)$ and $\mu[a,\infty) < \varepsilon$. In any case, there exists some interval in $\mathscr{R}_\varepsilon$ containing
$x$, and thus the collection of ``small'' intervals $\mathscr{R}_\varepsilon$ covers all of $X$.

Suppose there is some $n$ such that for every $A \in \mathcal{F}$ with positive measure, 
$\mu(A \cap T^{-n} A) > 0$. Consider $A \setminus T^{-n} A$; then
$\mu( \sq{A \setminus T^{-n} A} \cap T^{-n} \sq{A \setminus T^{-n} A } ) = \mu(\emptyset) = 0$, which implies that
$\mu( A \setminus T^{-n} A ) = \mu(T^{-n} A \setminus A) = 0$.\footnote{%
In general, we note for future reference that we have shown that if $A$ is any set of positive
measure in a measure-preserving dynamical system, 
and every positive measure subset $B \subseteq A$ satisfies $\mu(B \cap T^{-n} B) > 0$, then 
$\mu(A \setminus T^{-n} A) = \mu(T^{-n} A \setminus A) = 0$.}

Let us fix $\varepsilon>0$ and apply these observations to the collection $\mathscr{R}_\varepsilon$. Letting $I \in \mathscr{R}_\varepsilon$, 
we know that $\mu(T^{-n} I \setminus I) = 0$. Thus, for $\mu$-a.e.\ $x \in T^{-n} I$, we also have $x \in I$. Since
$I$ and $\genset{x,T^n x}$ are both closed intervals, we have
$\genset{x,T^n x} \subseteq I$, so that
$\mu\!\genset{x,T^{n} x} \le \mu(I) < \varepsilon$ for $\mu$-a.e.\ $x \in T^{-n} I$. That is to say,
\[
\mu\!\set{x \in X : T^{n} x \in I, \mu\!\genset{x,T^n x} \ge \varepsilon} = 0 .
\]
Since $\mathscr{R}_\varepsilon$ covers $X$, 
so also does $\set{T^{-n} I : I \in \mathscr{R}_\varepsilon}$, and we can write
\begin{align*}
\mu \set{x \in X : \mu\!\genset{x,T^n x} \ge \varepsilon}
&= \mu\!\p{ \bigcup_{I \in \mathscr{R}_\varepsilon} \set{x \in X : T^n x \in I,\; \mu\!\genset{x,T^n x} \ge \varepsilon}}\\
&\le \sum_{I \in \mathscr{R}_\varepsilon}\mu\!\set{x \in X : T^n x \in I,\; \mu\!\genset{x,T^n x} \ge \varepsilon} = 0 .
\end{align*}
Given that $\varepsilon$ was arbitrary, by letting this go to $0$ we find that in general it must be that
\[\mu\!\set{x \in X : \mu\!\genset{x,T^n x} > 0 } = 0 . \]
However, this implies that $\mu\!\genset{x,T^n x} = 0$ for $\mu$-a.e. $x$, contradicting our assumption that
the system was $n$-nigh aperiodic. Therefore, our supposition that $\mu(A \cap T^{-n} A) \neq 0$ for
every positive measure $A$ must be false, and consequently there must  exist  some set $F \in \mathcal{F}$ 
such that 
$\mu(E \cap T^{-n} E) = 0$, proving the first claim.

We now begin the second claim. Since
$\mu( U \triangle T^{-1} U ) = 0$ it follows that $\nu = \mu|_U$ is a $T$-preserved measure. 
Suppose that for every $A \in \mathcal{F}$, $\nu(A \cap T^{-n} A) > 0$. By the same reasoning, 
$\nu(T^{-n} A \setminus A) = 0$, and thus letting $I \in \mathscr{R}_\varepsilon$,
for $\nu$-a.e.\ $x \in I$, $T^n(x) \in I$, and so $\mu\!\genset{x,T^n x} < \varepsilon$,
so that
\[
\nu\!\set{x \in X : T^{n} x \in I, \mu\!\genset{x,T^n x} \ge \varepsilon} = 0 .
\]
The same reasoning as before allows us to conclude that $\nu\!\set{x \in X : \mu\!\genset{x,T^n x} \ge \varepsilon} = 0$
for every $\varepsilon > 0$, and thus $\nu\set{x \in X : \mu\!\genset{x,T^n x} > 0} = 0$. Then the complement of this set must
have $\nu$-measure $1$, so that we have
\begin{align*}
\nu\!\set{ x \in X : \mu\!\genset{x,T^n x} = 0 } &= 1 \\
\frac{\mu\!\p{U \cap \set{ x \in X : \mu\!\genset{x,T^n x} = 0 }}}{\mu(U)} &= 1\\
\mu\!\p{U \cap \set{ x \in X : \mu\!\genset{x,T^n x} = 0 } } &= \mu(U) >0 \\
\mu\!\set{ x \in X : \mu\!\genset{x,T^n x} = 0 } &> 0.
\end{align*}
This however contradicts our claim that $(X,\mathcal{F},\mu,T)$ was nigh aperiodic. Thus, there must exist some set
$E$ with positive $\nu$-measure such that 
\[ 
\nu(E \cap T^{-n} E) 
= 
\frac{ \mu(U \cap (E \cap T^{-n} E)) }{\mu(U)}
= 
\frac{ \mu([U \cap E] \cap T^{-n} [U \cap E]) }{\mu(U)}
= 
0 \]
and therefore it follows that for $E'=U \cap E$, $\mu(E' \cap T^{-n} E') =0$. The set $E'$ must also have positive $\mu$-measure
since
\[ 0 < \nu(E) 
= \frac{\mu(U \cap E)}{\mu(U)} = \frac{\mu(E')}{\mu(U)} \]
so that $E'$ satisfies the second claim.
\end{proof}

\begin{definition}
If $E \in \mathcal{F}$ is a set of positive measure such that $\mu(T^{-m} E \cap E) = 0$ for all $m=1,\cdots,n-1$, then 
we say that $E$ induces an \textbf{$n$-chain}.
\end{definition}

\begin{proposition}\label{prop2}
Let $(X,\mathcal{F}(<),\mu,T)$ be a nigh aperiodic shuffleable system. For any $n \in \N_+$, and any positive measure 
set $G \in \mathcal{F}$, there exists $F \subseteq G$ which induces an $n$-chain.
\end{proposition}
\begin{proof}
The proof will proceed by induction.
The case $n=1$ is trivial; let $F=G$. Suppose then that $G$ induces an $n-1$-chain, but does not
contain any subset inducing an $n$-chain.
By our observation from the previous proposition, if every positive measure subset $B \subseteq G$ satisfies 
$\mu(B \cap T^{-(n-1)} B) >0$, then $\mu(G \setminus T^{-(n-1)} G) = 0$, so that $G$ and $T^{-(n-1)} G$ 
are identical up
to a set of measure $0$. For any $j\in \set{0,\cdots,n-2}$, $\mu(T^{-j} G \triangle T^{-j-k(n-1)} G) = 0$. Thus,
\[
\mu\p{
\p{\bigcup_{j=0}^\infty T^{-j} G}
\triangle
\p{\bigcup_{j=0}^{n-2} T^{-j} G}
}
=0 .
\]
Then, we restrict $\mu$ to the set $G^\ast = \bigcup_{j=0}^{\infty} T^{-j} G$. 
By the Poincar\'e recurrence theorem, almost every point in $G$ recurs, and thus is in some
set $T^{-j} G$, so that 
\[ \mu(G^\ast \triangle T^{-1} G^\ast) = \mu\p{ G \setminus \bigcup_{j=1}^\infty T^{-j} G } = 0 \]
The set $G^\ast$ is measure theoretically equivalent to $\bigsqcup_{j=0}^{n-2} T^{-j} G$, and applying
the second claim of Proposition~\ref{prop1} to the new system $(X,\mathcal{F},\mu|_{G^\ast},T)$, we find a positive $\mu|_{G^\ast}$-measure 
subset $F' \subseteq G^\ast$ such that
$\mu|_{G^\ast}( F' \cap T^{-(n-1)} F' ) = 0$. In particular, there exists a $j \in \set{0,1,\cdots,n-2}$ such that
$\mu|_{G^\ast}(F' \cap T^{-j} G) > 0$; we translate this positive measure 
part of $F'$ back into $T^{-(n-1)}G$ and thus $G$ itself by taking $(n-j)+1$ preimages, and obtain the set 
\[
F \defeq G \cap T^{-(n-j)-1}( F' \cap T^{-j} G )
= 
T^{-(n-j)-2} (F') \cap G \cap T^{-(n-1)} G .
\]
We know $\mu|_{G^\ast}(F) = \mu|_{G^\ast}(T^{-j}(F)) = \mu|_{G^\ast}(F' \cap T^{-j} G) > 0$ and 
$\mu|_{G^\ast}(F \cap T^{-(n-1)} F) = 0$.
Since $G$ generates an $n-1$-chain, $G$, $T^{-1} G$, $\ldots$, $T^{-(n-2)} G$ are all disjoint equal measure components of $G^\ast$, and
so for any $A \in \mathcal{F}$,
\[
\mu|_{G^\ast}(A)= \frac{\mu(A \cap G^\ast)}{\mu(G^\ast)} = \frac{\mu(A \cap G^\ast)}{(n-1) \mu(G)},\]
which implies that
$\mu(F) > 0$ and $\mu(F \cap T^{-(n-1)} F) = 0$, since the corresponding claims in $\mu|_{G^\ast}$ hold also.
However, this implies that $F$ does in fact induce an $n$-chain, contrary to our previous assumption. Therefore, 
if $G$ is a set of positive measure inducing an $n-1$-chain, then $G$ contains a
subset $F \subseteq G$ of positive measure 
inducing an $n$-chain, and consequently $(X,\mathcal{F}(<),\mu,T)$ contains $n$-chains of all possible lengths
$n$.
\end{proof}

\begin{theorem}[Rokhlin lemma for shuffleable systems]
Let $(X,{\mathcal{F}(<)},\mu,T)$ be a nigh aperiodic shuffleable system. Then for any $n \in \N$ and $\varepsilon > 0$, there exists a set
$E$ inducing an $n$-chain such that 
\[
\mu\!\p{ \bigsqcup_{j=0}^{n-1} T^{-j} E } > 1-\varepsilon.
\]
\end{theorem}
We remind the reader that we build on the basic structure of the proof of the Rokhlin Lemma offered by Heinemann and 
Schmitt~\cite{Heinemann00}; their proof utilized the structure of a separable metric space with aperiodic dynamics 
to obtain the claim that $n$-chains exist for every $n$, as we have with Proposition~\ref{prop2}, but beyond this
the context is no longer necessary. 
\begin{proof}
Let $m \in \N$ be such that $m \ge n$ and $1/m < \varepsilon / (n-1)$. 
We know by Proposition~\ref{prop2} that there exist $m$-chains
in $\mathcal{F}$. We may equip the collection of sets in $\mathcal{F}$
which induce $m$-chains, $\mathscr{C}_m$, with the partial order
\[
F <_\mu F' \iff (F \subset F') \wedge (\mu(F) < \mu(F')).
\]
Suppose that $C$ is a chain in $\mathscr{C}_m$ in the sense of Zorn's 
lemma; we wish to show that $\mathscr{C}_m$ contains an upper bound
for the Zorn chain $C$. If $C$ contains an element $F^\ast$
with maximal measure for $C$, then we know that $\mu(F) < \mu(F^\ast)$ 
for all $F \in C$, so that $F^\ast$ is a unique
upper bound for $C$. Suppose
instead that $\sup_{F \in C} \mu(F) = M$, but $\mu(F) < M$ for every 
$F \in C$. For every $k \in \N$, we select some $F_k$ such that 
$\mu(F_k) > M-1/k$, and let $B = \bigcup_{k \in \N} F_k$. This is a 
countable union and therefore it resides in $\mathcal{F}$, and furthermore,
for any $F \in C$, there exists some $k$ such that $\mu(F) < \mu(F_k)$ 
and therefore also $F \subset F_k \subset B$. As a consequence, 
$B = \bigcup_{k \in \N} F_k = \bigcup_{F \in C} F$ is an upper bound for 
$C$ if it also induces an $m$-chain. If $B$ does not 
induce an $m$-chain,
then $\mu(B \cap T^{-j} B) > 0$ for some $j=1,2,\cdots,m-1$; but by continuity 
from below, this implies there must exist some $k_1,k_2 \in \N$ such
that $\mu(F_{k_1} \cap T^{-j} F_{k_2}) > 0$, and letting 
$k_0 = \max\set{k_1,k_2}$ this implies $\mu(F_{k_0} \cap T^{-j} F_{k_0}) > 0$, 
contradicting our assumption that $F_{k_0}$ itself induced an $m$-chain.
Thus, in any case, any chain $C \subseteq \mathscr{C}_m$
must accomodate an upper bound, and thus by Zorn's lemma, $\mathscr{C}_m$ 
contains maximal elements. Without loss of generality, let 
$F$ be a maximal element of $\mathscr{C}_m$, so that $F$ induces an $m$-chain, 
and no set containing $F$ with greater measure than $F$
induces an $m$-chain.

We claim that because $F$ is maximal in this sense, 
$\mu\p{\bigcup_{j=0}^\infty T^{-j} F} = 1$. Suppose that it were not; then the complement 
$Y = X \setminus \bigcup_{j=0}^\infty T^{-j} F$ must have positive measure; 
then by Proposition~\ref{prop2} there exists a positive measure $G \subseteq Y$ inducing
an $m$-chain. For every $k=0,1,\cdots,m-1$, 
\[
T^{-k} (G) \cap \bigcup_{j=0}^\infty T^{-j} F
\subseteq T^{-k} \p{
G \cap \bigcup_{j=0}^\infty T^{-j} F
}
= \emptyset .
\]
It follows that $\mu([F \cup G] \cap T^{-k} [F \cup G]) = 0$, and thus $F \cup G$ induces an $m$-chain.
Since $\mu(G) > 0$, we have $\mu(F) < \mu(F \cup G)$,
and thus $F <_\mu F \cup G$, in contradiction of the assumption that 
$F$ was maximal according to the partial order 
$<_\mu$. Thus, it follows that $\mu\p{\bigcup_{j=0}^\infty T^{-j} F } = 1$.
It  follows that almost every point in $X$ recurs to $F$ infinitely often.

We next define the sets $F^k = T^{-k} F \setminus \bigcup_{j=0}^{k-1} T^{-j} F$,
so that $F^k$ is the set of all $x \in X$ such that 
$T^k x \in F$ with $k$ the least natural number achieving this, and which
satisfy the recursive formula
\begin{align*}
F^{k+1} &= T^{-k-1}F \setminus \bigcup_{j=0}^{k} T^{-j} F
 = \p{T^{-k-1} F \setminus T^{-1} \p{ \bigcup_{j=0}^{k-1} T^{-j} F }} \setminus F\\
&= T^{-1} \p{F^{-k} \setminus \bigcup_{j=0}^{k-1} T^{-j} F} \setminus F
= T^{-1} F^k \setminus F .
\end{align*}
Utilizing their original definition, it is clear that the sets $F^k$ are pairwise 
disjoint, and that $\bigcup_{j=0}^N F^j = \bigcup_{j=0}^N T^{-j} F$,
so that $\mu\p{\bigsqcup_{j=0}^\infty F^j}=1$.

From this, we can tell that since $F^{k+1} \subseteq T^{-1} F^k$, more generally
we have $F^{k+j} \subseteq T^{-j} F^k$; and since 
$T^{-1} F^k \subseteq F^{k+1} \cup F$, so also 
\[ 
T^{-j} F^k \subseteq T^{-j+1} ( F^{k+1} \cup F ) 
\subseteq T^{-j+2} ( F^{k+2} \cup T^{-1} F \cup F )
\subseteq \cdots 
\subseteq F^{k+j} \cup \bigcup_{l=0}^{j-1} T^{-l} F .
\]
It follows that
\begin{equation}\label{eqn1}
F^{k+j} 
\subseteq T^{-j} F^k
\subseteq F^{k+j} \cup \bigcup_{l=0}^{j-1} T^{-l} F = F^{k+j} \cup \bigsqcup_{l=0}^{j-1} F^l  .
\end{equation}

We now assert that the set 
\[
E = \bigcup_{k=1}^\infty F^{kn-1}
\]
induces an $n$-chain. Suppose that for some $j \in \set{1,\cdots,n-1}$, we have
\[
x \in E \cap T^{-j} E .
\]
Then in particular, it must be that for some $p,q \in \N_+$, 
$x \in F^{p n -1} \cap T^{-j} (F^{q n -1})$. We know $F^{pn-1}$ is disjoint from $F^{qn-1+j}$, since $pn-1 \neq qn-1+j$
(which would imply $j = n(p-q)$, but $j \in \set{1,2,\cdots,n-1}$ which is disjoint from $n\Z$) and if $k_1 \neq k_2$ then
$F^{k_1} \cap F^{k_2} = \emptyset$ by construction. Since these two sets are disjoint, \eqref{eqn1} tells us that 
\[
x \in F^{pn-1} \cap T^{-j}(F^{qn-1})
\subseteq F^{pn-1} \cap \p{ F^{j+qn-1} \cup \bigsqcup_{l=0}^{j-1} F^l }
= F^{pn-1} \cap \bigsqcup_{l=0}^{j-1} F^l .
\]
By the disjointness of the different $F^k$, and that $j \le n-1$, 
we know that $pn-1 \in \set{0,\cdots,j-1} \subseteq \set{0,\cdots,n-2}$. But this would imply $pn \in \set{1,\cdots,n-1}$, which is
disjoint from $n\Z$, a contradiction. Thus, such an $x$ cannot exist, and so $E$ must induce an $n$-chain.

Finally, we estimate the measure of the $n$-chain induced by $E$, by using \eqref{eqn1} together with the fact that $F$ induces an
$m$-chain for $m \ge n$, and the fact that $\mu(F) \le 1/m < \varepsilon / (n-1)$, and find that it satisfies the theorem:
\begin{align*}
\mu\p{ \bigcup_{j=0}^{n-1} T^{-j} E }
&=   \mu \p{ \bigcup_{j=0}^{n-1} \bigcup_{k=1}^\infty T^{-j} F^{kn-1} }
 \ge \mu \p{ \bigcup_{j=0}^{n-1} \bigcup_{k=1}^\infty F^{kn-1+j} }\\
&=   \mu \p{ \bigcup_{i=n-1}^\infty F^i } 
 = 1 - \mu \p{\bigcup_{i=0}^{n-2} F^i}
 = 1 - \sum_{i=0}^{n-2} \mu(T^{-i} F) \\
&= 1-(n-1) \mu(F) 
\ge 1 - (n-1) \frac{\varepsilon}{n-1} = 1-\varepsilon .
\end{align*}
Thus, for arbitrary $n \in \N$ and $\varepsilon>0$, there exists a set $E$ inducing an $n$-chain which covers all but a 
subset of measure $\varepsilon$ of the space $X$.
\end{proof}

\section{Connections}
\label{sec3}
The contexts of the separable metric space of the Rokhlin lemma 
of Heinemann and Schmitt \cite{Heinemann00} and of our version for shuffleable space
are not so dissimilar as they may originally seem; 
a system on a separable metric space equipped also with completeness, a non-atomic measure, and without isolated points
must be conjugate to a shuffleable system. The connection develops through the following proposition, originally that of 
K.\ Kuratowski~\cite[p.~440]{Kuratowski} but refined into a shape more useful for our purposes by H.L.\ Royden~\cite{Royden88}.
\begin{proposition}
\label{prop:Royden}
\cite[Prop.~15.8]{Royden88}
Let $X$ be a complete separable metric space without isolated points, and let $\mathcal{I}$ be the set of 
irrational points in the unit interval. Then there is a one-to-one continuous map $\varphi$ of $\mathcal{I}$
onto $X$ such that $\varphi(O)$ is an $F_\sigma$ set in $X$ for each open subset $O$ of $\mathcal{I}$.
\end{proposition}
\begin{proposition}
Let $\Sigma = (X,\mathscr{B}(X),\mu,T)$ be a measure-preserving dynamical system with $X$ a complete separable metric space without isolated points,
$\mathscr{B}(X)$ the Borel $\sigma$-algebra on $X$, $\mu$ a non-atomic probability measure, and $T$ surjective. 
Then $\Sigma$ is measure-theoretically conjugate to a shuffleable system on the measurable space $(\mathcal{I},\mathscr{B}(\mathcal{I}))$.
\end{proposition}
\begin{proof}
We have assumed all of the conditions of Proposition~\ref{prop:Royden}, so there exists a bijective map 
$\varphi : \mathcal{I} \to X$ satisfying all of its
conclusions. In particular, since $\varphi$ is continuous, it is also measurable according to the Borel $\sigma$-algebras 
$\mathscr{B}(X)$ and $\mathscr{B}(\mathcal{I})$, and since the image of any open set $O$ is a $F_\sigma$ set, we also know
that it is forward measurable (if $E \in \mathscr{B}(\mathcal{I})$, then $\varphi(E) \in \mathscr{B}(X)$).
Thus, both $\varphi$ and $\varphi^{-1}$ are measurable.

We then define a measure on $\mathcal{I}$ by $\mu_0(E) = \mu(\varphi(E))$ and dynamics on $\mathcal{I}$ by 
$T_0(x) = \varphi^{-1}(T(\varphi(x)))$. We know $T_0$ preserves the measure $\mu_0$ since $T$ originally did:
\[
\mu_0(T^{-1}_0 E) = \mu( \varphi \varphi^{-1} T^{-1}  \varphi(E) )
= \mu(T^{-1} \varphi(E) )= \mu(\varphi(E)) = \mu_0(E)
\]
Thus, it follows that $(\mathcal{I},\mathscr{B}(\mathcal{I}),\mu_0,T_0)$ forms a measure-preserving dynamical system. It admits
the natural total order $<$, which generates its topology and also $\mathscr{B}(\mathcal{I})$, and it is conjugate through $\varphi$ to 
$(X,\mathscr{B}(X),\mu,T)$ which is non-atomic so that $\mu_0$ must also be non-atomic also. Further, it must also be surjective
since $T$ is surjective;
for any $y \in \mathcal{I}$, there must be $x \in X$ such that $T(x) = \varphi(y)$, and since $\varphi$ is bijective, there exists
$z \in \mathcal{I}$ such that $\varphi(z) = x$, so that $\varphi^{-1} T \varphi(z) = T_0(z) = y$. Thus,
$(\mathcal{I},\mathscr{B}(\mathcal{I}),\mu_0,T_0)$ is in fact a shuffleable system which is conjugate to $\Sigma$.
\end{proof}
As is, it is not clear that this system is necessarily nigh aperiodic.
Since Proposition~\ref{prop:Royden} is agnostic to measure theoretic concerns, and its structure does not immediately
lend itself to any obvious criteria which can guarantee this.

In 2016, Avila and Candela proved a version of the Rokhlin lemma for free measure-preserving actions on atomless standard 
probability space. Their version of the Rokhlin lemma connects to the shuffleable Rokhlin lemma in the cyclic subcase.
To see this connection, a few definitions are necessary for context.
\begin{definition}
Two probability spaces $(X,\mathcal{F},\mu)$ and $(Y,\mathcal{G},\nu)$ are \textbf{isomorphic} if there exists 
a bimeasurable bijection $f : X \to Y$ such that both $f$ and $f^{-1}$ are measure-preserving maps. If
instead there exist subsets $E_X$ and $E_Y$ with $\mu(E_X) = \nu(E_Y) = 0$ 
such that $X \setminus E_X$ is isomorphic to $Y \setminus E_Y$, then we say $(X,\mathcal{F},\mu)$ and $(Y,\mathcal{G},\nu)$
are \textbf{isomorphic mod $0$}.
\end{definition}

\begin{definition} \cite[Definition~2.1]{Sinai17}
$(X,\mathscr{L},\mu)$ is a \textbf{standard probability space}\footnote{It is sometimes permitted that $X$ also include an
atomic component, but this is explicitly excluded by Avila and Candela's theorem.}
if $X$ is isomorphic mod $0$ to the interval $[0,1]$ with the 
standard Lebesgue measure $m$.
\end{definition}

\begin{theorem} \cite[Theorem~1.2]{Avila16}
Let $\varepsilon>0$ and let $n \in \N^d$ be a $d$-tuple of positive integers. 
Then for every measure-preserving action $f$ of the monoid of $d$-tuples of nonnegative integers $\N_0^d$ on
a standard probability space such that for any $k,l \in \N_0^d$, $k \neq l$,
\begin{equation} \label{freeaction}
\mu(\set{x \in X : f_k(x) = f_l(x)}) = 0
\end{equation}
there exists an $n$-tower for $f$ of measure at least $1-\varepsilon$.
\end{theorem}
When we assume $d=1$, this reduces to a measure-preserving transform of an atomless standard probability space,
the more traditional form of the Rokhlin lemma, and the condition \eqref{freeaction} reduces to aperiodicity.

The context of a standard probability space is very similar to that of a shuffleable system, but not identical; 
the standard probability space must be on a complete $\sigma$-algebra, while shuffleable systems are equipped with a
Borel $\sigma$-algebra instead, which is not complete.
We therefore may extend our definition of shuffeable spaces to $\sigma$-algebras which are complete.
\begin{definition}
We say $\Sigma=(X,\mathscr{L}(<),\mu,T)$ is a \textbf{completed shuffleable system} if it is a probability
measure-preserving dynamical system where $\mathscr{L}$ is the completion of $\mathcal{F}(<)$ according to $\mu$,
$T$ is a measurable endomorphism, and $\mu$ is non-atomic.
\end{definition}

\begin{proposition}
If $(X,\mathscr{L}(X),\mu,T)$ is an aperiodic measure-preserving dynamical system on a
standard probability space, then there exists a full-measure subset $Y \subseteq X$  on which 
$(Y,\mathscr{L}(Y),\mu|_Y,T)$ is a nigh aperiodic completed shuffleable system.
\end{proposition}
\begin{proof}
Since $(X\mathscr{L}(X),\mu)$ is a standard probability space, there exists an isomorphism mod $0$
from $X$ to $[0,1]$,
$f : X \setminus E_X \to [0,1] \setminus E_I$. We can assume without loss of generality that 
$T(X \setminus E_X) = X \setminus E_X$; if necessary, we substitute $E'_X = \bigcup_{j=0}^{\infty} T^{-j} E_X$ for
$E_X$, and let $E'_I = [0,1] \setminus f(X \setminus E'_X)$.

The isomorphism mod $0$ between $X$ and $[0,1]$ then induces a total order $<_Y$ on $Y=X \setminus E_X$ via $f$, by
\[
a <_Y b \iff f(a) < f(b).
\]
By definition, $f$ is monotone according to the orders $<_Y$ and $<$.
Consequently, the 
intervals in $Y$ according to $<_Y$ are measurable since they are the preimages of intervals in $[0,1] \setminus E_I$ through
$f$, a bimeasurable map, and thus $<_Y$ generates $\mathcal{F}(<)$. In particular, $f(\genset{x,T^n x}) = \genset{f(x),f(T^n(x))}$,
so
$\mu(f(\genset{x,T^n x})) = m\!\genset{f(x),f(T^n(x))} = \abs{f(T^n(x)) - f(x)}$. Since $f$ is bijective, $f(T^n(x)) = f(x)$ iff
$T^n(x) = x$, and since $T$ is aperiodic, this only occurs on a set of zero measure. Therefore, in $Y$ according to the total
order $<_Y$,
\[
\mu\!\set{ x \in Y : \mu\!\genset{x,T^n x} = 0 } = 0
\]
and thus it follows that $T$ is nigh aperiodic on $Y$ also.
\end{proof}

Our proof of Rokhlin's lemma for shuffleable systems applies equally well to completed shuffleable systems;
applying this version of Rokhlin's lemma to $Y$ yields an alternative proof of Rokhlin's lemma for aperiodic dynamics on standard
probability spaces.


\end{document}